\documentclass{lmsMODIFIED}

\usepackage{amsmath, amscd, amsfonts, amssymb, graphicx}

\newtheorem{theorem}{Theorem}[section] 
\newtheorem{lemma}[theorem]{Lemma}     
\newtheorem{corollary}[theorem]{Corollary}
\newtheorem{proposition}[theorem]{Proposition}

\newnumbered{definition}[theorem]{Definition}
\newnumbered{remark}[theorem]{Remark}
\newnumbered{example}[theorem]{Example}
\newnumbered{conjecture}[theorem]{Conjecture}

\newcommand{\bb}[1]{\mathbb{#1}}
\newcommand{\abs}[1]{\vert#1\vert}

\renewcommand{\cal}[1]{\mathcal{#1}}
\newcommand{\roi}{\mathcal{O}}
\newcommand{\res}[1]{\overline{#1}}

\newcommand{\al}{\alpha}
\newcommand{\comment}[1]{}
\newcommand{\mult}[1]{#1^{\times}}
\newcommand{\w}{\omega}
\newcommand{\sub}[1]{{\mbox{\scriptsize {#1}}}}
\newcommand{\ul}[1]{{\underline{#1}}}

\DeclareMathOperator{\Char}{char}

\title[Fubini's theorem over a two-dimensional local field]{Fubini's theorem and non-linear changes of variables over a two-dimensional local field}

\author{Matthew Morrow}

\classno{11S80 (primary), 20G25, 20G05, 28C10 (secondary)}

\extraline{The author is supported by an EPSRC Ph.D. Plus fellowship at the University of Nottingham}

\begin{document}

\maketitle

\begin{abstract}
We consider non-linear changes of variables and Fubini's theorem for certain integrals over a two-dimensional local field. An interesting example is presented in which imperfectness of a positive characteristic local field causes Fubini's theorem to unexpectedly fail. The relationship to ramification theory is discussed.
\end{abstract}

\section{Introduction}

\subsection{Summary}
This paper considers the issue of Fubini's theorem and non-linear polynomial changes of variables for integration over a two-dimensional local field (that is, a complete discrete valuation field $F$ with residue field $\res{F}$ a usual local field). Integration over such fields was first considered by Fesenko \cite{Fesenko2003}, then by Kim and Lee \cite{Lee2005}. Hrushovski and Kazhdan \cite{Hrushovski2006} \cite{Hrushovski2008} then developed a very general integration theory on valued fields using model theory; their methods are valid for residue characteristic zero (and sufficiently large positive characteristic). The author \cite{Morrow2008} has reformulated Fesenko's work to provide a new perspective which appears to be useful in treating problems such as Fubini's theorem.

A satisfactory theory of translation-invariant integration on algebraic groups over two dimensional local fields should lead to significant advances in the representation theory of such groups. Kim and Lee consider this problem in \cite{Lee2005} and construct a translation-invariant measure on $\mbox{GL}_n(F)$. The methods of Hrushovski and Kazhdan yield all main results under the restriction on the characteristic. The author extended \cite{Morrow2008a} his methods of \cite{Morrow2008} to obtain a translation invariant integral on $\mbox{GL}_n(F)$ which lifts in a natural sense the usual Haar integral on $\mbox{GL}_n(\res{F})$.

A relation between this theory and quantum physics is evidenced by the two-dimensional local field $\bb{C}((X))$. Subspaces of this space such as $\bb{C}[X]$ may be identified with subspaces of the space of continuous paths in the plane. The Feynman integral is not understood rigorously (see \cite{Johnson2000} for discussion of the problems) and measure theory on these path spaces may help to explain it. Further evidence of the relations between quantum field theory and the measure on these fields may be found in sections 16, 18 of \cite{Fesenko2006} and example 5.6 of \cite{Morrow2008}.

To extend the approach in \cite{Morrow2008a} from $\mbox{GL}_n$ to an arbitrary algebraic group it is necessary to have a theory of integration on finite dimensional vector spaces over $F$ which behaves well under certain \emph{non-linear} changes of variable (for the $\mbox{GL}_n$ theory, linear changes of variable suffice). Moreover, for use in applications, it is essential that Fubini's theorem concerning repeated integrals is valid. This paper considers the problem of establishing whether the equality \[\int^F\int^F g(x,y-h(x))\,dydx=\int^F\int^F g(x,y-h(x))\,dxdy\] holds for appropriate functions $g$ and polynomials $h$. Moreover, the methods used appear to be suitable for changes of variables much more general than $(x,y)\mapsto (x,y-h(x))$.
In the first section of the paper we summarise the integration theory of \cite{Morrow2008}. Given a Haar integrable function $f:\res{F}\to\bb{C}$, let $f^0$ be the lift of $f$ to $F$ which vanishes outside $\roi_{F}$ and satisfies $f^0(x)=f(\res{x})$ for $x\in\roi_{F}$; functions expressible as finite sums of functions of the form \[x\mapsto f^0(\al x+a)\] for $\al\in\mult{F}$, $a\in F$, $f$ Haar integrable on $\res{F}$, are said to be integrable on $F$. A translation invariant $\bb{C}(X)$-valued integral $\int^F$ is defined for such functions.

In the second section we discuss the extension of this integration theory to $F\times F$. Given a Haar integrable function $f:\res{F}\times\res{F}\to\bb{C}$, let $f^0$ be the lift to $F\times F$ which vanishes outside $\roi_{F}\times\roi_{F}$. Taking translations and scaling of such function by $F\times F$ and $\mult{F}\times\mult{F}$ obtains a space of functions similar to the analogous space for $F$. However, is appears that there are many more functions $\Phi:F\times F\to \bb{C}$ for which \[\int^F\int^F\Phi(x,y)\,dxdy=\int^F\int^F\Phi(x,y)\,dydx;\] that is, for which both repeated integrals make sense according to the integration theory of $F$ and the two repeated integrals have the same value.

The third section describes the action of polynomials on $F$. Given a polynomial $h\in\roi_{F}[X]$, and a translated fractional ideal $b+t^A\roi_{F}\subseteq\roi_{F}$, we show how to write $\{x\in\roi_{F}:h(x)\in b+t^A\roi_{F}\}$ as a finite disjoint union of translated fractional ideals; here $t$ is a local parameter of $F$ as a discrete valuation field. If $a+t^c\roi_{F}$ is one of these translated fractional ideals, it is also important to understand the behaviour of the function \[h:a+t^c\roi_{F}\to b+t^A\roi_{F}.\] These results (indeed, the whole paper) are closely related to ramification theory of two-dimensional local fields through A.~Abbes and T.~Saito's theory of ramification for complete discrete valuation fields with imperfect residue field; a summary of their theory is given towards the end of section three, and further relations explained throughout the remainder of the paper.

The impetus of this paper is conjecture \ref{conjecture}, which we rapidly reduce to the following: if $f$ is a Schwartz-Bruhat function on $\res{F}\times\res{F}$, $f^0$ is defined as above, and $h\in F[X]$ is a polynomial, then surely \[\int^F\int^F f^0(x,y-h(x))\,dydx=\int^F\int^F f^0(x,y-h(x))\,dxdy.\] In section five the conjecture is shown to be true if $h$ is linear or if all coefficients of $h$ belong to $\roi_{F}$.

The technically difficult case of when $h$ contains coefficients not in $\roi_{F}$ is taken up in section six. Introduce a polynomial $q\in F[X]$ and integer $R<0$ by the three conditions $h(X)=h(0)+t^Rq(X)$, $q\in\roi_{F}[X]$, and $q\notin t\roi_{F}[X]$. We give explicit expressions for the integral of $\int^F f^0(x,y-h(x))\,dx$ in terms of the decomposition of sets of the form $\{x\in\roi_{F}:q(x)\in b+t^{-R}\roi_{F}\}$; the conjecture easily follows if $R=-1$ so long as $\res{q}$, the image of $q$ in $\res{F}[X]$, is not a purely inseparable polynomial. When $R<-1$ calculations become difficult, and the function $y\mapsto\int^F f^0(x,y-h(x))\,dx$ can fail to be integrable, meaning that the conjecture fails; however, we present examples suggesting that the space of integrable functions could be extended so as to remedy this deficit.

We then consider the possibility that $\res{F}$ has positive characteristic and $\res{q}$ is purely inseparable. When $R=-1$ it is shown that \[\int^F\int^F f^0(x,y-h(x))\,dydx=\int_{\res{F}}\int_{\res{F}} f(x,y)\,dydx\] but \[\int^F\int^F f(x,y-h(x))\,dxdy=0.\] So if $f$ has non-zero Haar integral over $\res{F}\times\res{F}$ then the conjecture drastically fails. This fascinating result provides an explicit example to show that the work of Hrushovski and Kazhdan really can fail in positive characteristic, and we discuss its relationship with ramification theory.

In the final section we summarise the results obtained and discuss future work.

\subsection{Acknowledgements et al.}
The bulk of this paper was written during the first half of my Ph.D. studies under I.~Fesenko, and I am grateful to him for many interesting discussions surrounding ad\`elic analysis. Later, G.~Yamashita kindly read the first version of this paper in detail and introduced me to the ramification theory of A.~Abbes and T.~Saito, which I began to study while travelling abroad in the final year of my PhD. These visits abroad would not have been possible without the generosity of the Cecil King Foundation and London Mathematical Society, who funded me through the Cecil King Travel Scholarship. 

Chapter $4$ of my thesis \cite{Morrow_thesis} consists mainly of the material of this paper, but with additional comments; the reader is recommended to consult that source in conjunction with the current paper.

\subsection{Notation}
$F$ is a complete discrete valuation field whose residue field $K=\res{F}$ is a local field ($\bb{R}$, $\bb{C}$, or non-archimedean). We fix a prime $t$ of $F$ and denote its ring of integers by $\roi_{F}$. The residue map $\roi_{F}\to\res{F}$ is denoted $x\mapsto\res{x}$; the discrete valuation is denoted $\nu:F\to\bb{Z}\cup\{\infty\}$. We fix a Haar measure on $K$.

\section{Integration on $F$}
In \cite{Morrow2008} a theory of integration on $F$ taking values in the field of rational functions $\bb{C}(X)$ is developed. We repeat here the definitions and main results:

\begin{definition}
Let $\cal{L}$ denote the space of complex-valued Haar integrable functions on $K=\res{F}$. Let $\cal{L}(F)$ denote the smallest $\bb{C}(X)$ space of all $\bb{C}(X)$-valued functions on $F$ with the following properties:
\begin{enumerate}
\item For all $f\in\cal{L}$, $\cal{L}(F)$ contains
\[f^0:F\to \bb{C}(X),\quad x\mapsto \begin{cases}
	f(\res{x}) & x\in\roi_{F},\\
	0 & \mbox{otherwise;}
	\end{cases}
\]
\item if $g\in\cal{L}(F)$ and $a\in F$ then $\cal{L}(F)$ contains $x\mapsto g(x+a)$;
\item if $g\in\cal{L}(F)$ and $\al\in\mult{F}$ then $\cal{L}(F)$ contains $x\mapsto g(\al x)$.
\end{enumerate}
Functions in $\cal{L}(F)$ are said to be \emph{integrable} on $F$. 
\end{definition}

\begin{remark}
For $f\in\cal{L}$, $a\in F$, and $n\in\bb{Z}$, the \emph{lift of $f$ at $a,n$}, denoted $f^{a,n}$, is the complex valued function of $F$ defined by $f^{a,n}(a+t^nx)=f^0(x)$ for all $x\in F$.

It is easy to see that $\cal{L}(F)$ is the $\bb{C}(X)$-space spanned by $f^{a,n}$ for $f\in\cal{L}$, $a\in F$, $n\in\bb{Z}$.
\end{remark}

The following is fundamental:
\begin{theorem}\label{theorem_main_properties_of_integral}
There is a unique $\bb{C}(X)$-linear functional $\int^F$ on $\cal{L}(F)$ which satisfies:
\begin{enumerate}
\item $\int^F$ lifts the Haar integral on $\res{F}$: for $f\in\cal{L}$, \[\int^F(f^0)=\int_K f(u)\,du.\]
\item Translation invariance: for $g\in\cal{L}(F)$, $a\in F$, \[\int^F g(x+a)\,dx=\int^F g(x)\,dx.\]
\item Compatibility with multiplicative structure: for $g\in\cal{L}(F)$, $\al\in\mult{F}$, \[\int^F g(\al x)\,dx=\abs{\al}^{-1}\int^F g(x)\,dx.\]
\end{enumerate}
Here the \emph{absolute value} of $\al\in\mult{F}$ is defined by $\abs{\al}=\left|\res{\al t^{-\nu(\al)}}\right|_KX^{\nu(\al)}$, and we have adopted the customary integral notation $\int^F(g)=\int^F g(x)\,dx$.
\end{theorem}
\begin{proof}
See \cite{Morrow2008}.
\end{proof}

\begin{remark}\label{remark_integral_of_lifted_function}
If $f^{a,n}$ is as in the previous remark then $\int^F f^{a,n}(x)\,dx=\int f(u)\,du\,X^n$.
\end{remark}

\begin{example}\label{eg_null_functions}
This example demonstrates some unusual behaviour of the integral.
\begin{enumerate}
\item We will first calculate the measure of $\roi_{F}$. Note that $\Char_{t\roi_{F}}=\Char_{\{0\}}^0$, and so
\begin{align*}
	\int^F \Char_{\roi_{F}}(x)\,dx
	&=\int^F \Char_{t\roi_{F}}(tx)\,dx\\
	&=X^{-1}\int^F \Char_{t\roi_{F}}(x)\,dx\\
	&=X^{-1}\int_K\Char_{\{0\}}(u)\,du\\
	&=0.
	\end{align*}
\item For a second example, introduce $g_1=\Char_{t\roi_{F}}$, and $g_2=-2\Char_{t\{x\in\roi_{F}:\,\res{x}\in S\}}$ where $S$ is a Haar measurable subset of $\res{F}$ of measure $1$; let $g=g_1+g_2$. These functions are all complex-valued so it makes sense to consider their absolute values (here $\abs{\cdot}$ denotes the usual absolute value on $\bb{C}$, not the absolute value on $F$ introduced in the previous theorem). The following identities hold:
\begin{align*}
	\int^F \abs{g(x)} \,dx=\int^F\abs{g_1(x)}\,dx&=0\\
	\int^F g(x)\,dx=\int^F \abs{g(x)-g_1(x)}\,dx&=-2X.
\end{align*}
Proofs of these identities may be found in \cite{Morrow2008} example 1.9.
\end{enumerate}
\end{example}
%
%
\section{Integration on $F\times F$}
In this section we briefly explain repeated integration on $F\times F$; more details are available in \cite{Morrow2008a}

\begin{definition}
A $\bb{C}(X)$-valued function $g$ on $F\times F$ is said to be \emph{Fubini} if and only if both its repeated integrals exist and are equal. That is, we require: 
\begin{enumerate}
\item for all $x\in F$, the function $y\mapsto g(x,y)$ is integrable, and then that the function $x\mapsto \int^F g(x,y)\,dy$ is also integrable;
\item for all $y\in F$, the function $x\mapsto g(x,y)$ is integrable, and then that the function $y\mapsto \int^F g(x,y)\,dx$ is also integrable;
\item $\int^F\int^F g(x,y)\,dxdy = \int^F\int^F g(x,y)\,dydx$.
\end{enumerate}
Similarly, an integrable complex valued function $f$ on $K\times K$ will be called \emph{Fubini} if and only if both its repeated integrals exist and are equal.
\end{definition}

\begin{remark}
Recall that the existence and equality of the repeated integrals of a complex valued function on $K\times K$ does not imply its integrability on $K\times K$ (see e.g. \cite[example 8.9c]{Rudin1987}) which is why we have separately imposed that condition. However, in our applications we will restrict to well enough behaved functions for this subtle problem to be irrelevant.

Fubini's theorem implies that almost all (in the sense of failing off a set of measure of zero) the horizontal and vertical sections of any integrable function on $K\times K$ are integrable. Therefore any integrable function on $K\times K$ differs off a null set from some Fubini function. However, there is no satisfactory theory of lifting null sets from $K$ to $F$, so we restrict attention to Fubini functions on $K\times K$.

Any function in the Schwartz-Bruhat space of $K\times K$ is Fubini; recall that if $K$ is archimedean these are the smooth functions of rapid decay at infinity, and if $K$ is non-archimedean these are the locally constant functions of compact support.
\end{remark}

The main properties of the collection of Fubini functions on $F\times F$ are the following:
\begin{proposition}\label{proposition_properties_of_fubini_functions}
The collection of Fubini functions on $F\times F$ is a $\bb{C}(X)$-space with the following properties:
\begin{enumerate}
\item If $g$ is Fubini on $F\times F$, then so is $(x,y)\mapsto g(\al_1 x+a_1,\al_2 y+a_2)\,X^n$ for any $a_i\in F$, $\al_i\in\mult{F}$, $n\in\bb{Z}$, with repeated integral \begin{align*}\int^F\int^F g(\al_1 x+a_1,&\al_2 y+a_2)\,X^n\,dxdy\\&=\abs{\al_1}^{-1}\abs{\al_2}^{-1}\int^F\int^F g(x,y)\,dxdy X^n.\end{align*}
\item If $f$ is Fubini on $K\times K$, then \[f^0(x,y):=\begin{cases} f(\res{x},\res{y})&x,y\in\roi_{F},\\0&\mbox{otherwise,}\end{cases}\] is Fubini on $F\times F$, with repeated integral \[\int^F\int^F f^0(x,y)\,dxdy=\int_K\int_K f(u,v)\,dudv.\]
\end{enumerate}
\end{proposition}
\begin{proof}
The proof is straightforward; it may be found in \cite{Morrow2008a}.
\end{proof}

\begin{remark}\label{remark_repeated_integral_of_lifted_function}
The proposition implies that if $f$ is Fubini on $K\times K$, $a_1,a_2\in F$, $n_1,n_2\in\bb{Z}$, then the function $g=f^{(a_1,a_2),(n_1,n_2)}$ of $F\times F$ defined by \[f^{(a_1,a_2),(n_1,n_2)}(a_1+t^{n_1}x,a_2+t^{n_2}y)=f^0(x,y)\] for all $x,y\in F$ is Fubini. The function $g$ is said to be the \emph{lift of $f$ at $(a_1,a_2),(n_1,n_2)$}.

Proposition \ref{proposition_properties_of_fubini_functions} implies \[\int^F\int^F g(x,y)\,dxdy=\int^F\int^F g(x,y)\,dydx=\int_K\int_K f(u,v)\,dudv\,X^{n_1+n_2}.\]
\end{remark}


\section{Decompositions and ramification}\label{section_decompositions_and_ramificiation}

In this section we investigate the action of polynomials on $F$ and the relations with ramification theory; the results hold for any Henselian discrete valuation field $F$ with infinite residue field:

\begin{lemma}
Suppose $h(X)$ is a polynomial over $F$, that $a+t^c\roi_{F}$, $b+t^A\roi_{F}$ are two translated fractional ideals, and that $h(a+t^c\roi_{F})\subseteq b+t^A\roi_{F}$. Then there is a unique polynomial $\psi\in K [X]$ which gives a commutative diagram
\[\begin{CD}
a+t^c\roi_{F} @>h>> b+t^A\roi_{F} \\
@V{a+t^c x\mapsto\res{x}}VV  @VV{b+t^Ax\mapsto\res{x}}V  \\
 K  @>>\psi>  K .\\
\end{CD}\]
Moreover, $\deg\psi\le\deg h$.
\end{lemma}
\begin{proof}
There is certainly at most one function $\psi$ making the diagram commute; but $ K $ is an infinite field so if two polynomials are equal as functions then they are equal as polynomials. So there can be at most one polynomial $\psi$.

We may write $h(a+t^cX)=h(a)+t^RH(X)$ where $H\in\roi_{F}[X]$ is a polynomial with integer coefficients, no constant term, and with non-zero image in $ K [X]$ (i.e. not all coefficients of $H$ are in $t\roi_{F}$). We clearly have a commutative diagram
\[\begin{CD}
a+t^c\roi_{F} @>h>> h(a)+t^R\roi_{F} \\
@V{a+t^c x\mapsto\res{x}}VV  @VV{h(a)+t^Rx\mapsto\res{x}}V  \\
 K  @>>\res{H}>  K ,\\
\end{CD}\]
where $\res{H}$ denotes the image of $H$ in $ K [X]$.

If $A>R$ then the inclusion $h(a+t^c\roi_{F})\subseteq b+t^A\roi_{F}$ implies $\res{H}$ is everywhere equal to $\res{(b-h(a))t^{-R}}$; but $ K $ infinite then implies $\res{H}$ is a constant polynomial and hence is zero (since $H$ has no constant term), which is a contradiction. Hence $A\le R$, and we may easily complete the proof:

If $A=R$ then the desired commutative diagram is
\[\begin{CD}
a+t^c\roi_{F} @>h>> b+t^A\roi_{F} \\
@V{a+t^c x\mapsto\res{x}}VV  @VV{b+t^Ax\mapsto\res{x}}V  \\
 K  @>>\res{H}+\res{(h(a)-b)t^{-A}}>  K ,\\
\end{CD}\]
where the lower horizontal map is the function $u\mapsto\res{H}(u)+\res{(h(a)-b)t^{-A}}$. If $A<R$ then the desired diagram is
\[\begin{CD}
a+t^c\roi_{F} @>h>> b+t^A\roi_{F} \\
@V{a+t^c x\mapsto\res{x}}VV  @VV{b+t^Ax\mapsto\res{x}}V  \\
 K  @>>\res{(h(a)-b)t^{-A}}>  K ,\\
\end{CD}\]
where the lower horizontal map is the constant function $u\mapsto\res{(h(a)-b)t^{-A}}$.
\end{proof}

\begin{definition}
Suppose $h(X)$ is a polynomial over $F$, that $a+t^c\roi_{F}$, $b+t^A\roi_{F}$ are two translated fractional ideals, and that $h(a+t^c\roi_{F})\subseteq b+t^A\roi_{F}$. The unique polynomial $\psi\in K [X]$ which gives a commutative diagram
\[\begin{CD}
a+t^c\roi_{F} @>h>> b+t^A\roi_{F} \\
@V{a+t^c x\mapsto\res{x}}VV  @VV{b+t^Ax\mapsto\res{x}}V  \\
 K  @>>\psi>  K .\\
\end{CD}\]
is said to be the \emph{residue field approximation} of $h$ with respect to the translated fractional ideals.
\end{definition}

\begin{remark}
Regarding the previous definition, the translated fractional ideals will usually be clear from the context. The constant term of $\psi$ depends not only the sets $a+t^c\roi_{F}$ and $b+t^A\roi_{F}$, but on the representatives $a,b$ we choose.

When drawing the diagram above, we will henceforth omit the vertical maps, even though they do depend on the choice of $a,b$. We will also follow the habit used in the previous lemma of denoting a constant function on $ K $ by the value it assumes.
\end{remark}

Much of this paper is concerned with the problem of explicitly decomposing the preimage of a set under a polynomial and describing the resulting residue field approximations. Here is a example to illustrate the ideas:

\begin{example}
Set $q(X)=X^3+X^2+t^2$ and assume $\mbox{char}\, K \neq 2$. The aim of this example is to give explicit descriptions of the sets $\{x\in\roi_{F}:q(x)\in t^A\roi_{F}\}$ for $A=2,3$, as well as all associated residue field approximations.

Direct calculations easily show that if $x\in\roi_{F}$, then $q(x)\in t^2\roi_{F}$ if and only if $x\in t\roi_{F}$ or $x\in-1+t^2\roi_{F}$. Further, the residue field approximations are
\begin{align*}\begin{CD}
t\roi_{F} @>q>>t^2\roi_{F} \\
@VVV  @VVV  \\
 K  @>>X^2+1>  K \\
\end{CD}
&\qquad&
\begin{CD}
-1+t^2\roi_{F} @>q>>t^2\roi_{F} \\
@VVV  @VVV  \\
 K  @>>X+1>  K \\
\end{CD}\end{align*}

Similarly, if we suppose $x\in\roi_{F}$ then $q(tx)\in t^3\roi_{F}$ if and only if $x^2+1\in t\roi_{F}$; and $q(-1+t^2x)\in t^3\roi_{F}$ if and only if $x\in-1+t\roi_{F}$.

If $ K $ contains a square root of $-1$, let $i$ denote a lift of it to $\roi_{F}$; then \[\{x\in\roi_{F}:q(x)\in t^3\roi_{F}\}=it+t^2\roi_{F}\sqcup-it+t^2\roi_{F}\sqcup-1-t^2+t^3\roi_{F},\] with residue field approximations
\[\begin{array}{lcc}
\begin{CD}
it+t^2\roi_{F} @>q>>t^3\roi_{F} \\
@VVV  @VVV  \\
 K  @>>2iX+\res{(i^2+1)t^{-1}}-\res{i}>  K 
\end{CD}
&\,\,&
\begin{CD}
-it+t^2\roi_{F} @>q>>t^3\roi_{F} \\
@VVV  @VVV  \\
 K  @>>-2iX+\res{(i^2+1)t^{-1}}+\res{i}>  K 
\end{CD}\\
&&\\
\begin{CD}
-1-t^2+t^3\roi_{F} @>q>>t^3\roi_{F} \\
@VVV  @VVV  \\
 K  @>>X>  K 
\end{CD}
&&\\
\end{array}\]

If $ K $ does not contain a square root of $-1$, then $\{x\in\roi_{F}:q(x)\in t^3\roi_{F}\}=-1-t^2+t^2\roi_{F}$, with the residue field approximation given by the third diagram above.
\end{example}

We now turn to generalising the example to an arbitrary polynomial; for later applications to integration the following results will allow us to reduce calculations to the residue field, where we change variable according to the residue field approximation polynomials for example, and then return to $F$.

The first decomposition result treats the non-singular part of the polynomial, and is really just a rephrasing of Hensel's lemma:

\begin{proposition}\label{proposition_non_singular_decomposition}
Let $q(X)$ be a polynomial with coefficients in $\roi_{F}$, of degree $\ge1$ and with non-zero image in $ K [X]$; let $b\in F$. 
\begin{enumerate}
\item Suppose that $q(a)=b$ for some $a\in\roi_{F}$ and that $\res{q}'(\res{a})\neq 0$. Then for any $A\ge 1$, \[\{x\in\roi_{F}\,:\,\res{x}=\res{a}\mbox{ and }q(x)\in b+t^A\roi_{F}\}=a+t^A\roi_{F},\] and the residue field approximation is `multiplication by $\res{q}'(\res{a})$':
\[\begin{CD}
a+t^A\roi_{F} @>q>> b+t^A\roi_{F} \\
@VVV  @VVV  \\
 K  @>>\res{q}'(\res{a})X>  K \\
\end{CD}\]
\item Let $\w_1,\dots,\w_r$ be the simple (i.e. $\res{q}'(\w_i)\neq 0$) solutions in $ K $ to $\res{q}(X)=\res{b}$; let $\check{\w}_i$ be a lift by Hensel of $\w_i$ to $\roi_{F}$; that is, $q(\w_i)=b$. Then for any $A\ge 1$, \[\{x\in\roi_{F}\,:\;\res{q}'(\res{x})\neq0\mbox{ and }q(x)\in b+t^A\roi_{F}\} = \bigsqcup_{i=1}^r\check{\w}_i+t^A\roi_{F}.\]
\end{enumerate}
\end{proposition}
\begin{proof}
(i) is essentially contained in the proof of Hensel's lemma and so we omit it. (ii) easily follows.
\end{proof}

We now consider the singular part, which is much more interesting and will be the root of future difficulties:

\begin{proposition}\label{proposition_singular_decomposition}
Let $q(X)$ be a polynomial with coefficients in $\roi_{F}$, of degree $\ge1$ and with non-zero image in $ K [X]$; let $b\in F$. For $A\ge 1$ there is a decomposition \[\{x\in\roi_{F}\,:\,\res{q}'(\res{x})= 0\mbox{ and }q(x)\in b+t^A\roi_{F}\}=\bigsqcup_{j=1}^N a_j+t^{c_j}\roi_{F}\] (assuming this set is non-empty i.e. that $\res{q}(X)-\res{b}$ has a repeated root in $K$), where $a_1,\dots,a_N\in\roi_{F}$, and $c_1,\dots,c_N\ge 1$ are positive integers.
\end{proposition}
\begin{proof}
First suppose $A=1$. Let $a_1,\dots,a_N$ be lifts to $\roi_{F}$ of the distinct solutions in $K$ to $\res{q}(X)=\res{b}$ and $\res{q}'(X)=0$, and set $c_j=1$ for each $j$. Then the required decomposition is \[\bigsqcup_{j=1}^N a_j+t^{c_j}\roi_{F}.\] We now determine the residue field approximation of $q$ on each $a_j+t^{c_j}\roi_{F}$ as it will be used later in corollary \ref{corollary_R=-1}. So, for each $j$, consider the Taylor expansion \[q(a_j+tX)=q(a_j)+q'(a_j)tX+q_2(a_j)t^2X^2+\dots+q_d(a_j)t^dX^d\] where $d=\mbox{deg}\,q$. But $q'(a_j)\in t\roi_{F}$ implies $q(a_j+tx)\in q(a_j)+t^2\roi_{F}$ for all $x$ in $\roi_{F}$, which is to say that
\[\begin{CD}
a_j+t^{c_j}\roi_{F} @>q>> b+t\roi_{F} \\
@VVV  @VVV  \\
 K  @>>\res{(q(a_j)-b)t^{-1}}>  K \\
\end{CD}\]
commutes, where the lower horizontal map is constant i.e. each residue field approximation associated to the decomposition is constant.

We now suppose $A>1$ and proceed by induction on $A$. Let $u_1,\dots,u_N\in K $ be the distinct solutions to $\res{q}(X)=\res{b}$ and $\res{q}'(X)=0$, and write \[W_j=\{x\in\roi_{F}\,:\,\res{x}=u_j\mbox{ and }q(x)\in b+t^A\roi_{F}\}\] for $j=1,\dots,N$. Since \[\{x\in\roi_{F}\,:\,\res{q}'(\res{x})= 0\mbox{ and }q(x)\in b+t^A\roi_{F}\}=\bigsqcup_{j=1}^N W_j,\] it is enough to decompose each $W_j$ in the required manner, so we now fix a value of $j$, writing $W=W_j$ and $u=u_j$. 

If $W$ is empty then we are done; else $u$ has a lift to $a\in\roi_{F}$ such that $q(a)\in b+t^A\roi_{F}$, and we now fix such an $a$. Using the same Taylor expansion as above, there exist $\rho\ge1$ and $Q\in \roi_{F}[X]$ such that $q(a+tX)=q(a)+t^{\rho}Q(X)$ and $\res{Q}(X)\neq0$; in fact, $q'(a)\in t\roi_{F}$ implies $\rho\ge2$, though we will not use this. Therefore \[W=a+t\{x\in\roi_{F}\,:\,Q(x)\in(b-q(a))t^{-\rho}+t^{A-\rho}\roi_{F}\},\] but also note that \[(b-q(a))t^{-\rho}+t^{A-\rho}\roi_{F}=t^{A-\rho}((b-q(a))t^{-A}+\roi_{F})=t^{A-\rho}\roi_{F}\] by choice of $a$. Therefore $W=a+t\{x\in\roi_{F}\,:\,Q(x)\in t^{A-\rho}\roi_{F}\}$, and it becomes clear how the induction should proceed.

In fact, we must consider three cases, depending on the relative magnitudes of $\rho$ and $A$:

\begin{enumerate}
\item $A-\rho<0$. Then $\{x\in\roi_{F}\,:\,Q(x)\in t^{A-\rho}\roi_{F}\}=\roi_{F}$ and $Q(\roi_{F})\subseteq\roi_{F}\subset t^{A-\rho}\roi_{F}$; therefore the residue field approximation is constant, given by the diagram
\[\begin{CD}
\roi_{F} @>Q>> t^{A-\rho}\roi_{F} \\
@VVV  @VVV  \\
 K  @>>0>  K \\
\end{CD}\]
This implies $W=a+t\roi_{F}$ with a constant residue field approximation:
\[\begin{CD}
a+t\roi_{F} @>q>>b+t^A\roi_{F} \\
@VVV  @VVV  \\
 K  @>>\res{(q(a)-b)t^{-A}}>  K .\\
\end{CD}\]

\item $A-\rho=0$. Again, $\{x\in\roi_{F}\,:\,Q(x)\in t^{A-\rho}\roi_{F}\}=\roi_{F}$; the residue field approximation is clearly
\[\begin{CD}
\roi_{F} @>Q>>\roi_{F} \\
@VVV  @VVV  \\
 K  @>>\res{Q}>  K \\
\end{CD}\]
Therefore $W=a+t\roi_{F}$, with residue field approximation
\[\begin{CD}
a+t\roi_{F} @>q>>b+t^A\roi_{F} \\
@VVV  @VVV  \\
 K  @>>\res{Q}(X)+\res{(q(a)-b)t^{-A}}>  K .\\
\end{CD}\]

\item $A-\rho>0$. Here we may use the inductive hypothesis and proposition \ref{proposition_non_singular_decomposition} to write \[\{x\in\roi_{F}\,:\,Q(x)\in t^{A-\rho}\roi_{F}\}=\bigsqcup_i d_i+t^{e_i}\roi_{F},\] with residue field approximations $\psi_i(X)$, say:
\[\begin{CD}
d_i+t^{e_i}\roi_{F} @>Q>>t^{A-\rho}\roi_{F} \\
@VVV  @VVV  \\
 K  @>>\psi_i>  K \\
\end{CD}\]
Therefore $W=\bigsqcup_i a+d_it+t^{e_i+1}\roi_{F}$, with residue field approximations
\[\begin{CD}
a+d_it+t^{e_i+1}\roi_{F} @>q>>b+t^A\roi_{F} \\
@VVV  @VVV  \\
 K  @>>\psi_i(X)+\res{(q(a)-b)t^{-A}}>  K .\\
\end{CD}\]
\end{enumerate}
\end{proof}

For $q(X)$ as in the previous two propositions, these two decomposition results completely describe $\{x\in\roi_{F}:q(x)\in b+t^A\roi_{F}\}$ in terms of $\le(\deg q)^A$ translated fractional ideals equipped with polynomial residue field approximations. Moreover, the proof of the second result gives some insight into how structure of the polynomial $q$ effects the resulting residue field approximations. For applications beyond those described in this paper, it will be necessary to better understand how the decomposition varies with $b$ and $A$. For small $A$ we have the following result:

\begin{lemma}\label{lemma_depth_1_and_2}
Let $q(X)$ be a polynomial with coefficients in $\roi_{F}$, of degree $\ge1$ and such that $q'$ has non-zero image in $ K [X]$; let $A=1$ or $2$. There are finitely many $b_1,\dots,b_m\in\roi_{F}$ such that if $b\in\roi_{F}$ and $\{x\in\roi_{F}:q(x)\in b+t^A\roi_{F},\;\res{q}'(\res{x})=0\}$ is non-empty, then $b\equiv b_i\mod t^A\roi_{F}$ for some $i\in\{1,\dots,m\}$.
\end{lemma}
\begin{proof}
First suppose $A=1$. Then $\{x\in\roi_{F}:q(x)\in b+t^A\roi_{F},\res{q}'(\res{x})=0\}$ being non-empty implies that $\res{b}$ is the image under $\res{q}$ of one of the finitely many roots of $\res{q}'$.

Now suppose that $A=2$. Then the argument is just the same as for $A=1$, except it is important to observe the following: if $a_1,a_2\in\roi_{F}$ are equal modulo $t\roi_{F}$, and $\res{q}'(\res{a}_i)=0$ for $i=1,2$, then $q(a_1)=q(a_2)\mod t^2\roi_{F}$. This follows from the Taylor expansion and the fact that $q'(a_i)\in t\roi_{F}$.
\end{proof}

Decomposition results similar to the previous ones are common in model theory; for example, in the theory of algebraically closed valued fields \cite{Robinson1977}, every definable subset of the field is a finite disjoint union of points and `Swiss cheeses'. Further, these decompositions are related to ramification theory and rigid geometry through A.~Abbes and T.~Saito's \cite{Abbes2002} \cite{Abbes2003} ramification theory for complete discrete valuation fields with imperfect residue field, about which we will now say some more.

Until Abbes and Saito' work it was a significant open problem to systematically generalise the classical ramification theory for complete discrete valuation fields with perfect residue field to the imperfect residue field situation; some alternative approaches are due to J.~Borger \cite{Borger2004} \cite{Borger2004a}, K.~Kato \cite{Kato1989} \cite{Kato1994}, and I.~Zhukov \cite{Zhukov2000} \cite{Zhukov2003}. Geometrically, the importance of this lies in the following situation. If $\phi:S_1\to S_2$ is a finite morphism between smooth, projective surfaces, over a field $k$ which is allowed to be perfect, then according to Zariski's `purity of the branch locus', the ramification of $\phi$ occurs along curves. Let $B\subset S_1$ be an irreducible curve with generic point $y$, and set \[K(S_1)_y=\mbox{Frac}\widehat{\roi_{S_1,y}}\,;\] this is a complete discrete valuation field whose residue field is $k(B)$. Moreover, we have a finite extension \[K(S_1)_y/K(S_2)_{\phi(y)},\] whose ramification properties reflect the local ramification of $\phi$ along $B$. But $k(B)$ will be imperfect and $K(S_1)_y/K(S_2)_{\phi(y)}$ may have an inseparable residue field extension.

We now give a summary of the basics of Abbes and Saito's theory. There is a more extensive overview by L.~Xiao \cite{Xiao2007}. Let $L/M$ be a finite, Galois extension of complete discrete valuation fields with arbitrary residue fields. Then $\roi_L$ is a complete intersection algebra over $\roi_M$ (since they are both regular local rings) and we may therefore write \[\roi_L=\roi_M[T_1,\dots,T_n]/\langle f_1,\dots,f_n\rangle\] for a regular sequence $f_1,\dots,f_n$.

Now, for any real $a\ge 1$, one introduces the rigid space \[X_{L/M}^a=\{\ul{x}\in (M^\sub{alg})^n:\nu_M(x_i)\ge0\mbox{ all }i,\,\nu_M(f_i(\ul{x}))\ge0\mbox{ all }i\},\] where $\nu_M:M^\sub{alg}\to\bb{Q}\cup\{\infty\}$ is the extension of the discrete valuation on $M$. Let $\pi_0(X_{L/M}^a)$ denote the set of connected components (in the sense of rigid geometry) of this space; if $n=1$, which one may assume if $\res{L}/\res{M}$ is separable, then it follows from our decomposition results \ref{proposition_non_singular_decomposition} and \ref{proposition_singular_decomposition} that these components are balls. As $a\to\infty$, $X_{L/M}^a$ will consist of $|L:M|$ small balls; conversely, $X_{L/M}^0$ is a single large ball. A central idea of Abbes and Saito's theory is to analyse the behaviour of $X_{L/M}^a$ as $a$ varies; in particular, when it breaks into $|L/M|$ balls. This will soon be made precise.

The natural action of the absolute Galois group $\mbox{Gal}(M^\sub{alg}/M)$ on $X_{L/M}^a$ induces an action on $\pi_0(X_{L/M}^a)$, which then factors transitively through $G=\mbox{Gal}(L/M)$.

\begin{remark}\label{remark_ramification_theory_in_perfect_case}
To motivate what follows, let us briefly suppose that $M$ has perfect residue field. Let $\eta_{L/M}:[-1,\infty]\to[-1,\infty]$ be the inverse of the Hasse-Herbrand function $\psi_{L/M}$, relating the upper and lower ramification filtrations by $G^a=G_{\psi_{L/M}(a)}$ (see \cite{Fesenko2002} or \cite{Neukirch1999}). Then it is not hard to prove: \begin{quote}{\em For $a\ge -1$, $\sigma\in G$ acts trivially on $\pi_0(X_{L/M}^{\eta_{L/M}(a)+1})$ if and only if $\sigma\in G_a$.}\end{quote} (A nice sketch is given in \cite{Xiao2007}). So, for any $a\ge -1$, the kernel of the action of $G$ on $\pi_0(X_{L/M}^{a+1})$ is $G^a$.
\end{remark}

Abbes and Saito take this final observation in this remark as the definition of the upper filtration in their theory:

\begin{definition}
Let $L/M$ be a finite, Galois extension of complete discrete valuation fields. The {\em upper ramification filtration} on $G=\mbox{Gal}(L/M)$, is defined, for $a\ge -1$, by \[G^a=\ker(G\to\mbox{Aut}(\pi_0(X_{L/M}^{a+1}))).\]
\end{definition}

Starting from this definition of the upper ramification filtration, Abbes and Saito develop fully a ramification theory for $M$. Xiao has augmented their work by establishing the Hasse-Arf integrality theorem for certain conductors \cite{Xiao2008a} \cite{Xiao2008b}.

Now suppose that $M=F$ is our two-dimensional local field. As discussed above, a problem which will appear later is the variation of the sets $\{x\in\roi_{F}:q(x)\in b+t^A\roi_{F}\}$ as $b$ varies across $\roi_F$. This is related, albeit mysteriously, with understanding the ramification of extensions of the form $F(\al)/F$, where $\al$ is a root of $q(X)-b$. In fact, we will soon assume that $q(0)=0$, and an alternative viewpoint is that both the integration and ramification theories study the monodromy action of $(x,y)\mapsto y-q(x)$ around zero. Unfortunately, there are currently no precise comparison results between the integration and ramification theories.


\section{Non-linear changes of variables}

In this section we investigate the behaviour of Fubini functions on $F\times F$ under certain non-linear changes of variables. More precisely, we consider the following:

\begin{conjecture}\label{conjecture}
Let $a_1,a_2\in F$, $n_1,n_2\in\bb{Z}$, and let $h(X)$ be a polynomial over $F$. Then for any Schwartz-Bruhat function $f$ on $K\times K$, letting $g=f^{(a_1,a_2),(n_1,n_2)}$ be the lift of $f$  at $(a_1,a_2),(n_1,n_2)$, the function \[\Phi(x,y)=g(x,y-h(x))\] is Fubini on $F\times F$, with repeated integral equal to that of $f$.
\end{conjecture}

The conjecture is false in the generality in which we have stated it, though an important special case has already been done treated in a previous paper of the author:

\begin{theorem}\label{theorem_linear_case}
With notation as in the conjecture, if $\deg h\le 1$ then the conjecture is true.
\end{theorem}
\begin{proof}
The main result of \cite{Morrow2008a}, for a two-dimensional space, is that for any $\tau\in GL_2(F)$, the function $(x,y)\mapsto g(\tau(x,y))$ is Fubini on $F\times F$. The conjecture is a special case of that result when $\deg h=1$; in fact, it essentially follows from lemma 3.7 of \cite{Morrow2008a}

If $\deg h=0$ then the conjecture follows from translation invariance of the integral; see proposition \ref{proposition_properties_of_fubini_functions} and remark \ref{remark_repeated_integral_of_lifted_function}.
\end{proof}

Because of the previous theorem, we will have in mind polynomials $h(X)$ of degree at least $2$, though our results are equally valid for lower degree. We will be interested in conditions on the data $a_1,a_2,n_1,n_2,h$ such that the conjecture is true for all Schwartz-Bruhat functions $f$. We assign to the data two invariants as follows:

\begin{definition}
Let $a_1,a_2,n_1,n_2,h$ be data for the conjecture, and write $h(a_1+t^{n_1}X)=h(a_1)+t^Rq(X)$, where $R\in\bb{Z}$, $q\in\roi_{F}[X]$, and the image of $q$ in $ K [X]$ is non-zero. Note that $q(0)=0$.

The \emph{depth} and \emph{normalised polynomial} associated to the data are defined to be $R-n_2$ and $q(X)$ respectively.
\end{definition}

A summary of what we know about the validity of the conjecture, classified by the depth and normalised polynomial, may be found in section \ref{section_summary}. The sense in which the depth and normalised polynomial are invariants, and why they are useful, is given by the following lemma in which we reduce the conjecture to a special case:

\begin{lemma}\label{lemma_reduction}
Fix a polynomial $q\in\roi_{F}[X]$ with nonzero image in $ K [X]$ and no constant term, and an integer $R\in\bb{Z}$. Then the following are equivalent:
\begin{enumerate}
\item the conjecture is true for all data $a_1,a_2,n_1,n_2,h$ with depth $R$ and normalised polynomial $q$;
\item the conjecture is true for all data of the form $0,0,0,0,h$ with depth $R$, normalised polynomial $q$, and such that $h(0)=0$;
\item for all Schwartz-Bruhat functions $f$ on $K\times K$, the function \[(x,y)\mapsto f^0(x,y-t^Rq(x))\] is Fubini;
\item for all Schwartz-Bruhat functions $f$ on $K\times K$, the following hold: for each $y\in F$, the function $x\mapsto f^0(x,y-t^Rq(x))$ is integrable, then that $y\mapsto\int^Ff^0(x,y-t^Rq(x))\,dx$ is integrable, and finally that \[\int^F\int^F f^0(x,y-t^Rq(x))\,dxdy=\int_K\int_K f(u,v)\,dudv.\]
\end{enumerate}
\end{lemma}
\begin{proof}
Clearly (i)$\Rightarrow$(ii). The only data satisfying the conditions of (ii) are $0,0,0,0,t^Rq$, and so (ii)$\Leftrightarrow$(iii).

(iii)$\Rightarrow$(i): So assume (iii), letting $a_1,a_2,n_1,n_2,h$ be data for the conjecture with depth $R$ and normalised polynomial $q$. Let $f$ be Schwartz-Bruhat on $K\times K$ and write $g=f^{(a_1,a_2),(n_1,n_2)}$. Note that $h(a_1+t^{n_1}X)=h(a_1)+t^{R+n_2}q(X)$, and that therefore for all $x,y\in F$,
\begin{align*}
g(a_1+t^{n_2}x,a_2+t^{n_2}y-h(a_1+t^{n_1}x))
	&=f^0(x,y-t^{-n_2}h(a_1+t^{n_1}x))\\
	&=f^0(x,(y-t^{-n_2}h(a_1))-t^Rq(x)).
\end{align*}

By (iii), this final function of $(x,y)$ differs from a Fubini function by translation. So $(x,y)\mapsto g(x,y-h(x))$ differs from a Fubini function only by translation and scaling, and hence is itself Fubini, by proposition \ref{proposition_properties_of_fubini_functions}. Therefore we have proved (i).

(iii)$\Leftrightarrow$(iv): First note that for any $x\in F$, the function $y\mapsto f^0(x,y-t^Rq(x))$ is just the translation of $y\mapsto f^0(x,y)$ by $t^Rq(x)$; since $f^0$ is Fubini this is integrable, and translation invariance of the integral implies \[\int^F f^0(x,y-t^Rq(x))\,dy=\int^F f^0(x,y)\,dy.\] But as a function of $x$ this is integrable, again since $f^0$ is Fubini, and \[\int^F\int^F f^0(x,y-t^Rq(x))\,dydx=\int^F\int^F f^0(x,y)\,dydx.\] Now by remark \ref{remark_repeated_integral_of_lifted_function} and Fubini's theorem for $K\times K$, \[\int^F\int^F f^0(x,y)\,dydx=\int_K\int_K f(u,v)\,dudv.\]

By the definition of a Fubini function, it now follows that $(x,y)\mapsto f^0(x,y-t^Rq(x))$ is Fubini if and only if the $dxdy$ repeated integral is well-defined and equals $\int_K\int_K f(u,v)dudv$, which is precisely what is stated in (iv).
\end{proof}

With these reductions at hand it is straightforward to establish the conjecture in the case of non-negative depth:

\begin{theorem}\label{theorem_R_positive}
Let $a_1,a_2,n_1,n_2,h$ be data for the conjecture, and suppose that the associated depth is non-negative. Then the conjecture is true.
\end{theorem}
\begin{proof}
By the reductions, we suppose that $q\in\roi_{F}[X]$ is a polynomial with no constant term and non-zero image in $ K [X]$, that $R\ge0$ is an integer, and we will prove condition (iv) of the lemma above. Write $h(X)=t^Rq(X)$, and let $f$ be Schwartz-Bruhat on $K\times K$.

The assumption on $R$ implies that all coefficients of $h$ are integral, and for $y\in F$ we have
\[ \{x\in \roi_{F}\,:\,y-h(x)\in\roi_{F}\}
	=\begin{cases}
		\roi_{F}&y\in\roi_{F},\\
		\varnothing & y\notin\roi_{F}.
	\end{cases}
\]
Hence if $y\in\roi_{F}$, we see that $x\mapsto f^0(x,y-h(x))$ is the lift of \[u\mapsto f(u,\res{y}-\res{h}(u))\] at $0,0$, where $\res{h}$ is the image of $h$ in $ K [X]$. If $y\notin \roi_{F}$, then $f^0(x,y-h(x))=0$ for all $x$ in $F$.

Integrating with respect to $x$ therefore obtains \[\int^F f^0(x,y-h(x))\,dx=\begin{cases} \int_K f(u,\res{y}-\res{h}(u))\,du & y\in \roi_{F}, \\ 0 & y\notin \roi_{F},\end{cases}\] which simply says that $y\mapsto \int^F f^0(x,y-h(x))\,dx$ is the lift of \[v\mapsto \int_K f(u,v-\res{h}(u))\,du\] at $0,0$.

Hence we may integrate with respect to $y$ to get
\begin{eqnarray*}\int^F\int^F f^0(x,y-h(x))\,dxdy
	&=\int_K\int_K f(u,v-\res{h}(u))\,dudv\\
	&=\int_K\int_K f(u,v-\res{h}(u))\,dvdu
\end{eqnarray*}
where the second line follows from the first by Fubini's theorem on $K\times K$. The result now follows by translation invariance of the measure on $K$ and lemma \ref{lemma_reduction}.
\end{proof}

\begin{remark}
This is convenient opportunity for an additional remark on ramification theory. According to our discussion at the end of section \ref{section_decompositions_and_ramificiation}, the conjecture relates to the monodromy or ramification of $q(X)-Y$ around the origin. Non-negative depth corresponds to the unramified situation, so the previous theorem could be stated as ``The conjecture is true in the unramified case''.
\end{remark}


\section{Negative depth}\label{section_negative_depth}
Having reduced the problem as far as possible and treated the relatively easy case, we discuss the case of negative depth in this section and the following section \ref{section_purely_insep}.

For this section and the next we fix the following notation: $R<0$ a negative integer as the depth; a polynomial $q\in\roi_{F}[X]$ without constant term and with non-zero image in $ K [X]$ as the normalised polynomial; and a Schwartz-Bruhat function $f$ on $K\times K$. Write $\Phi$ for the function of $F\times F$ given by $\Phi(x,y)=f^0(x,y-t^Rq(x))$, and $\res{q}$ for the image of $q$ in $ K [X]$.

In this section, we also assume that $\res{q}$ does not have everywhere vanishing derivative; since $\res{q}$ is non-zero and without constant term, this condition can only fail to be satisfied if $ K $ has positive characteristic $p$ and $q(X)$ is a purely inseparable polynomial i.e. a polynomial in $X^p$. We shall drop this assumption in section \ref{section_purely_insep} and see that conjecture \ref{conjecture} fails for such highly singular $q$.

We will study the conjecture for data of depth $R$ and normalised polynomial $q$ through condition (iv) of lemma \ref{lemma_reduction}. We will establish various conditions under which the conjecture holds.

\begin{remark}\label{remark_unferocious_case}
From the point of view of monodromy, we are now in the unferocious case; recall that an extension of complete discrete valuation fields is said to be {\em unferocious} so long as the residue field extension is separable. The author believes that this terminology is due to I.~Zhukov \cite{Zhukov2000}.
\end{remark}

Introduce two sets: the non-singular part of $q$ \[W_\sub{ns}=\{x\in\roi_{F}\,:\,\res{q}'(\res{x})\neq 0\}=\{x\in\roi_{F}\,:\,q'(x)\in \mult{\roi_{F}}\},\] and the singular part \[W_\sub{sing}=\{x\in\roi_{F}\,:\,\res{q}'(\res{x})=0\}=\{x\in\roi_{F}\,:\,q'(x)\in t\roi_{F}\}.\] By our assumption on $q$, the non-singular part $W_\sub{ns}$ is non-empty. The corresponding singular and non-singular parts of $\Phi$ are the restriction of $\Phi$ to these sets extended by zero elsewhere: \begin{align*} \Phi_\sub{ns}&=\Phi\;\Char_{W_\sub{ns}\times F}\\ \Phi_\sub{sing}&=\Phi\;\Char_{W_\sub{sing}\times F}.\end{align*} Note that $\Phi=\Phi_\sub{ns}+\Phi_\sub{sing}$.

The singular and non-singular parts are treated separately. Using the decomposition result \ref{proposition_non_singular_decomposition}, we will now explicitly evaluate $x\mapsto\Phi_\sub{ns}(x,y)$ for any $y\in F$:

\begin{proposition}
For all $y\in F$, the function $x\mapsto\Phi_\sub{ns}(x,y)$ is integrable, and $y\mapsto \int^F\Phi_\sub{ns}(x,y)\,dx$ is the lift of \[v\mapsto \sum_{\substack{\w\in K  \\\res{q}(\w)=v \\ \res{q}'(\w)\neq 0}} \int_K f(\w,-\res{q}'(\w)u)\,du\,X^{-R}\] at $0,R$; the sum is taken over all simple solutions $\w$ to $\res{q}(\w)=v$.

Moreover, this function $y\mapsto \int^F\Phi_\sub{ns}(x,y)\,dx$ is integrable on $F$, with \[\int^F\int^F\Phi_\sub{ns}(x,y)\,dxdy=\int_K\int_K f(\w,u)\,d\w du.\]
\end{proposition}
\begin{proof}
Firstly, for $y\notin t^R\roi_{F}$, we have $\Phi(x,y)=0$ for all $x\in F$. Now fix $y=t^Ry_0\in t^R\roi_{F}$.

Then for $\Phi_\sub{ns}(x,y)$ to be non-zero, $x$ must belong to
\begin{align*}
\{x\in W_\sub{ns}\,:\,y-t^Rq(x)\in\roi_{F}\}
	&=\{x\in W_\sub{ns}\,:\,q(x)\in y_0+t^{-R}\roi_{F}\}\\
	&=\{x\in\roi_{F}\,:\,q(x)\in y_0+t^{-R}\roi_{F},\,\res{q}'(\res{x})\neq0\}\\
	&=\bigsqcup_{i=1}^r\check{\w}_i+t^{-R}\roi_{F},
\end{align*}
where $\check{\w}_i$ are lifts by Hensel of the simple solutions $\w_i$ in $ K $ to $q(\w)=\res{y}_0$ and the decomposition is provided by the decomposition result \ref{proposition_non_singular_decomposition}; that proposition also implies that there are commutative diagrams
\[\begin{CD}
\check{\w}_i+t^{-R}\roi_{F} @>q>> y_0+t^{-R}\roi_{F} \\
@VVV  @VVV  \\
 K  @>>\res{q}'(\w_i)X>  K .\\
\end{CD}\]

So we write $\Phi_\sub{ns}(x,y)=\sum_{i=1}^r g_i(x)$, where $g_i$ is the restriction of $x\mapsto\Phi_\sub{ns}(x,y)$ to $\check{\w}_i+t^{-R}\roi_{F}$, extended by zero elsewhere; if $x=\check{\omega}_i+t^{-R}x_0$ belongs to $\check{\w}_i+t^{-R}\roi_{F}$ then the commutative diagram implies \[\Phi_\sub{ns}(x,y)=g_i(x)=f(\w_i,-\res{q}'(\w_i)\res{x}_0).\] Therefore $g_i$ is the lift of the Haar integrable function \[u\mapsto f(\w_i,-\res{q}'(\w_i)u)\] at $\check{\w}_i,-R$, the integral of which is \[\int^F g_i(x)\,dx=\int_K f(\w_i,-\res{q}'(\w_i)u)\,du\,X^{-R}\] by remark \ref{remark_integral_of_lifted_function}. By linearity, $x\mapsto\Phi_\sub{ns}(x,y)$ is integrable, with \[\int^F\Phi_\sub{ns}(x,y)\,dx=\sum_{i=1}^r \int_K f(\w_i,-\res{q}'(\w_i)u)\,du\,X^{-R}\tag{$\ast$}.\]

The previous paragraph considered a fixed value of $y=t^Ry_0$ in $t^R\roi_{F}$. We now consider the integral ($\ast$) as a function of $y$; that is, \[y\mapsto\int^F\Phi_\sub{ns}(x,y)\,dx.\] Recall that $\w_1,\dots,\w_r$ are the simple solutions in $K$ to $q(\w)=\res{y}_0$. So we may rewrite the integral as \[\int^F \Phi_\sub{ns}(x,y)\,dx=\sum_\w \int_K f(\w,-\res{q}'(\w)u)\,du\,X^{-R},\] where the sum is over the finitely many $\w$ in $K$ which satisfy $\res{q}(\w)=\res{y}_0$ and $\res{q}'(\w)\neq0$.

Finally, by the appendix, the function $v\mapsto\sum_{\substack{\w:\,\res{q}(\w)=v\\\res{q}'(\w)\neq0}}\int_K f(\w,-\res{q}'(\w)u)\,du$ is in fact Haar integrable on $K$ with integral \[\int_K\sum_{\substack{\w:\,\res{q}(\w)=v\\\res{q}'(\w)\neq0}}\int_K f(\w,-\res{q}'(\w)u)\,dudv=\int_K\int_K f(\w,u)\,d\w du.\] Therefore $y\mapsto \int^F\Phi_\sub{ns}(x,y)\,dx$ is integrable on $F$, with \[\int^F\int^F\Phi_\sub{ns}(x,y)\,dxdy=\int_K\int_K f(\w,u)\,d\w du.\]
\end{proof}

The proposition has an immediate corollary:

\begin{corollary}\label{corollary_nowhere_singular}
If $\res{q}'(X)$ is no-where vanishing on $ K $, then $\Phi$ is Fubini.
\end{corollary}
\begin{proof}
If $\res{q}'(X)$ has no roots in $ K $, then $\Phi=\Phi_\sub{ns}$, so the previous proposition and lemma \ref{lemma_reduction} imply $\Phi$ is Fubini.
\end{proof}

More generally, the proposition reduces the problem to showing that the singularities of $\res{q}$ give no contribution to the integrals:

\begin{corollary}\label{corollary_reduction_to_singular_part}
The function $\Phi$ is Fubini if and only if the following hold: for each $y\in F$, the function $x\mapsto\Phi_\sub{sing}(x,y)$ is integrable, then that $y\mapsto\int^F\Phi_\sub{sing}(x,y)\,dx$ is integrable, and finally that \[\int^F\int^F \Phi_\sub{sing}(x,y)\,dxdy=0.\]
\end{corollary}
\begin{proof}
This follows immediately from the identity $\Phi=\Phi_\sub{ns}+\Phi_\sub{sing}$, the previous proposition, lemma \ref{lemma_reduction}, and linearity.
\end{proof}

We may verify the first requirement of corollary \ref{corollary_reduction_to_singular_part} using the decomposition result \ref{proposition_singular_decomposition}:

\begin{proposition}\label{proposition_integral_of_singular_sections}
For each $y\in F$, the function $x\mapsto\Phi_\sub{sing}(x,y)$ is integrable, and we have the following explicit descriptions of its integral:

If $y\notin t^R\roi_{F}$, or if $y\in t^R\roi_{F}$ but $\{x\in\roi_{F}\,:\,q(x)\in t^{-R}y+t^{-R}\roi_{F},\,\res{q}'(\res{x})=0\}$ is empty, then $\int^F\Phi_\sub{sing}(x,y)\,dx=0$.

Otherwise we have $y\in t^R\roi_{F}$ and write \[\{x\in\roi_{F}\,:\,q(x)\in t^{-R}y+t^{-R}\roi_{F},\,\res{q}'(\res{x})=0\}=\bigsqcup_{j=1}^N a_j+t^{c_j}\roi_{F},\] where the decomposition (which depends on $y$) is provided by the decomposition result \ref{proposition_singular_decomposition}; let $\psi_j\in K[X]$ for $j=1,\dots,N$ denote the corresponding residue field actions i.e.
\[\begin{CD}
a_j+t^{c_j}\roi_{F} @>q>> t^{-R}y+t^{-R}\roi_{F} \\
@VVV  @VVV  \\
 K  @>>\psi_j>  K .\\
\end{CD}\]
commutes. Then \[\int^F\Phi_\sub{sing}(x,y)\,dx={\sum_j^{}}' \int_K f(\res{a}_j,-\psi_j(u))\,du\,X^{c_j},\] where the summation $\sum'$ is over those $j\in\{1,\dots,N\}$ for which $\psi_j$ is not a constant polynomial.
\end{proposition}
\begin{proof}
By the definition of a lifted function, $f^0$ vanishes off $\roi_{F}\times\roi_{F}$. So if $\{x\in\roi_{F}\,:\,q(x)\in t^{-R}y+t^{-R}\roi_{F},\,\res{q}'(\res{x})=0\}$ is empty for some $y$ then $x\mapsto\Phi_\sub{sing}(x,y)$ is everywhere zero and hence integrable; note that this set is certainly empty if $y\notin t^R\roi_{F}$.

Now fix $y=t^Ry_0\in t^R\roi_{F}$ for the remainder of the proof. Then for $x\in F$, $\Phi_\sub{sing}(x,y)$ vanishes unless $x$ belongs to
\begin{align*}
\{x\in W_\sub{sing}\,:\,y-t^Rq(x)\in\roi_{F}\}
	&=\{x\in W_\sub{sing}\,:\,q(x)\in y_0+t^{-R}\roi_{F}\}\\
	&=\{x\in\roi_{F}\,:\,q(x)\in y_0+t^{-R}\roi_{F},\,\res{q}'(\res{x})=0\}\\
	&=\bigsqcup_{j=1}^N a_j+t^{c_j}\roi_{F},
\end{align*}
where the decomposition is as in the statement of the proposition; let $\psi_j$ be the corresponding residue field approximations. Denote by $g_j$ the restriction of $x\mapsto\Phi_\sub{sing}(x,y)$ to $a_j+t^{c_j}\roi_{F}$, extended by zero elsewhere. We shall now prove that each $g_j$ is an integrable function. Indeed, $g_j$ vanishes off $a_j+t^{c_j}\roi_{F}$, and if $x=a_j+t^{c_j}x_0\in a_j+t^{c_j}\roi_{F}$, then
\begin{align*}
g_j(x)&=f^0(a_j+t^{c_j}x_0, t^Ry_0-t^Rq(a_j+t^{c_j}x_0))\\
	&=f^0(a_j+t^{c_j}x_0,t^R(y_0-q(a_j+t^{c_j}x_0))\\
	&=f(\res{a_j+t^{c_j}x_0},\res{t^R(y_0-q(a_j+t^{c_j}x_0))})\\
	&=f(\res{a}_j,-\psi_j(\res{x}_0))
\end{align*}
by definition of the residue field approximation $\psi_j$. Therefore $g_j$ is a lifted function: it is the lift of $u\mapsto f(\res{a}_j,-\psi_j(u))$ at $a_j,c_j$. Further, since we assumed $f$ is Schwartz-Bruhat, this function of $u$ is Haar integrable on $K$ \emph{so long as $\psi_j$ is not constant}, and therefore $g_j$ is integrable on $F$, with \[\int^F g_j(x)\,dx=\int_K f(\res{a}_j,-\psi_j(u))\,du\,X^{c_j}.\] However, if $\psi_j$ is a constant polynomial, then $g_j=g_j(a_j)\Char_{a_j+t^{c_j}\roi_{F}}$, which is integrable with zero integral by example \ref{eg_null_functions}.

By linearity, $x\mapsto\Phi_\sub{sing}(x,y)$ is integrable, with \[\int^F\Phi_\sub{sing}(x,y)\,dx={\sum_j^{}}' \int_K f(\res{a}_j,-\psi_j(u))\,du\,X^{c_j},\] as required. We emphasise again that the decomposition $a_j,c_j,\psi_j$ which we have used to express the integral \emph{depends on $y$}.
\end{proof}

\begin{corollary}\label{corollary_R=-1}
If $R=-1$ then $\Phi$ is Fubini.
\end{corollary}
\begin{proof}
Looking at the proof of decomposition result \ref{proposition_singular_decomposition}, we see that if $R=-1$ (i.e. $A=1$ in the notation of that result), then all the residue field approximations are constant. So by the previous proposition, $\int^F\Phi_\sub{sing}(x,y)\,dx=0$ for all $y\in F$. Corollary \ref{corollary_reduction_to_singular_part} implies $\Phi$ is Fubini.
\end{proof}

By proposition \ref{proposition_integral_of_singular_sections} we now have a well defined function $y\mapsto\int^F\Phi_\sub{sing}(x,y)\,dx$; to establish the validity of the conditions of corollary \ref{corollary_reduction_to_singular_part} the next step is to prove that this function of $y$ is integrable. The complication in establishing its integrability is that we lack explicit information on the variation of the sets \[\{x\in\roi_{F}\,:\,q(x)\in y_0+t^{-R}\roi_{F},\;\res{q}'(\res{x})=0\}\] as $y_0$ runs though $\roi_{F}$.

We now present some results and calculations which reveal considerable insight into why $y\mapsto\int^F\Phi_\sub{sing}(x,y)\,dx$ can in fact fail to be integrable. We shall also give evidence that this phenomenon is merely a result of the integration theory not yet being sufficiently developed.

\begin{proposition}\label{proposition_finite_image}
Assume that there exist $b_1,\dots,b_m\in\roi_{F}$ such that if $b\in\roi_{F}$ and $\{x\in\roi_{F}:q(x)\in b+t^{-R}\roi_{F},\;\res{q}'(\res{x})=0\}$ is non-empty, then $b\equiv b_i\mod t^{-R}\roi_{F}$ for some $i\in\{1,\dots,m\}$. Note that this is satisfied if $R=-1$ or $-2$, by corollary \ref{lemma_depth_1_and_2}.

Then $y\mapsto\int^F\Phi_\sub{sing}(x,y)\,dx$ is a finite sum of lifts of functions of the form \[v\mapsto \int_K f(a,-\psi(u)-v)\,du\,X^{c}\] for $\psi\in K[X]$ non-constant, $a\in K$, and $c\ge 1$.
\end{proposition}
\begin{proof}
Let $b_1,\dots,b_m$ be as in the statement of the proposition; we also assume that $b_1,\dots,b_m$ are distinct modulo $t^{-R}\roi_{F}$.

By proposition \ref{proposition_integral_of_singular_sections}, if $y\in F$ is not in $b_it^R+\roi_{F}$ for some $i$, then $\int^F\Phi_\sub{sing}(x,y)\,dx=0$. So letting $G_i$ be the restriction of $y\mapsto\int^F\Phi_\sub{sing}(x,y)\,dx$ to $b_it^R+\roi_{F}$, extended by zero elsewhere, we have an equality of functions of $y$: \[\int^F\Phi_\sub{sing}(x,y)\,dx=\sum_{i=1}^mG_i(y).\]

For convenience of notation, we now fix some $i$ and write $G=G_i$, $b=b_i$. Write \[\{x\in\roi_{F}:q(x)\in b+t^{-R}\roi_{F},\;\res{q}'(\res{x})=0\}=\bigsqcup_{j=1}^Na_j+t^{c_j}\roi_{F},\] with residue field approximations $\psi_j$. We claim that $G$ is the lift of \[v\mapsto\stackrel{N}{{\sum_{j=1}}'} \int_K f(\res{a}_j,-\psi_j(u)-v)\,du\,X^{c_j}\] at $b^{-R},0$ (the sum ${\sum}'$ is restricted to those $j$ such that $\psi_j$ is not constant). So suppose $y=bt^R+y_0\in bt^R+\roi_{F}$. Then of course $yt^{-R}+t^{-R}\roi_{F}=b+t^{-R}\roi_{F}$, and so \[\{x\in\roi_{F}:q(x)\in yt^{-R}+t^{-R}\roi_{F},\res{q}'(\res{x})=0\}=\bigsqcup_{j=1}^Na_j+t^{c_j}\roi_{F},\] with the residue field approximations of this decomposition given by
\[\begin{CD}
a_j+t^{c_j}\roi_{F} @>q>> t^{-R}y+t^{-R}\roi_{F} \\
@VVV  @VVV  \\
 K  @>>\psi_j(X)-\res{y}_0>  K .\\
\end{CD}\]
Proposition \ref{proposition_integral_of_singular_sections} implies \[G(y)=\int^F\Phi_\sub{sing}(x,y)\,dx=\stackrel{N}{{\sum_{j=1}}'} \int_K f(\res{a}_j,-\psi_j(u)-\res{y}_0)\,du\,X^{c_j},\] proving the claim, and completing the proof.
\end{proof}

\begin{remark}\label{remark_nonsense}
Suppose that the assumption of the previous proposition is satisfied. Then to establish integrability of $y\mapsto\int^F\Phi_\sub{sing}(x,y)\,dx$ and prove it has zero integral, it is enough to prove that for any $a\in K$, non-constant $\psi\in K[X]$, and $g$ Schwartz-Bruhat on $K$, the lift of $v\mapsto \int_K g(-\psi(u)-v)\,du$ at $0,0$ is integrable and has zero integral; let $G$ denote this function of $F$, that is,
\begin{align*}
	G:F&\to \bb{C}\\
	y&\mapsto \begin{cases}
		\int_K g(-\psi(u)-\res{y})\,du & y\in\roi_{F}, \\
		0 & \mbox{otherwise.}
		\end{cases}
\end{align*}

\emph{Then $G$ may not be integrable on $F$}. Indeed, it is not hard to show that if $G$ were to belong to $\cal{L}(F)$, the space of integrable functions, then $G$ would be the lift at $0,0$ of a Haar integrable function on $K$; this Haar integrable function would then have to be $v\mapsto \int_K g(-\psi(u)-v)\,du$, but the arguments to follow reveal that this function is Haar integrable if and only if $g=0$.

We now offer the following nonsense argument for why $G$ should be integrable, and why $\int^F G(y)\,dy$ should be zero. As a lifted function, we evaluate the integral of $G$ by theorem \ref{theorem_main_properties_of_integral} to give \[\int^F G(y)\,dy=\int_K\int_K g(-\psi(u)-v)\,dudv\] and then apply Fubini's theorem for $K$ and translation invariance of the integral to deduce
\begin{align*}
\int^F G(y)\,dy
	&=\int_K\int_K g(-\psi(u)-v)\,dvdu\\
	&=\int_K\int_K g(-v)\,dvdu\\
	&=\int_K\,du\int_Kg(v)dv.
\end{align*}
At this point it is clear why our arguments are not valid: the function $v\mapsto \int_K g(-\psi(u)-v)\,du$ is not integrable on $K$. However, we may apply similar nonsense to the function $\Char_{\roi_{F}}$, which is the lift of $\Char_{K}$, to deduce \[\int^F\Char_{\roi_{F}}(x)\,dx=\int_K\,du.\] Finally, example \ref{eg_null_functions}(i) implies $\int^F\Char_{\roi_{F}}(x)\,dx=0$ and so
\begin{align*}
\int^F G(y)\,dy
	&=\int_K\,du\int_Kg(v)dv\\
	&=\int^F\Char_{\roi_{F}}(x)\,dx\int_Kg(v)dv\\
	&=0.
\end{align*}

It should be possible to extend the measure theory on $F$ so that these manipulations become rigorous. The key idea is that from the vantage point of $F$, the residue field $K$ truly has zero measure, as used above; so one expects Fubini's theorem on $K$ to hold for certain functions which, though not Haar integrable, are integrable in some sense after imposing the condition $\int_K\,du=0$. Once this is properly incorporated into the measure, the theory should become considerably richer. It should also yield new methods to treat divergent integrals on $K$ by lifting them to $F$, applying Fubini theorem there, and then pulling the results back down to $K$; this would be a refreshing contrast to the main techniques so far, which have centred around reducing integrals on $F$ down to $K$.
\end{remark}

\begin{example}\label{example_nonsense}
Now we treat an example of depth $-3$ in which the assumption of proposition \ref{proposition_finite_image} is not satisfied. We assume $R=-3$, $q(X)=X^2$, and $\mbox{char}\, K \neq 2$. The decompositions required for the proposition are given by
\begin{align*}
\{x\in\roi_{F}:q(x)&\in b+t^{-R}\roi_{F},\;\res{q}'(\res{x})=0\}\\
	&=\{x\in\roi_{F}:x^2\in b+t^3\roi_{F},\;\res{x}=0\}\\
	&=\begin{cases}
		\varnothing & b\notin t^2\roi_{F},\\
		\varnothing & b\in t^2\roi_{F}\mbox{ but }\res{t^{-2}b}\notin K^2,\\
		b^{1/2}+t^2\roi_{F} \sqcup -b^{1/2}+t^2\roi_{F} & b\in t^2\roi_{F}\mbox{ and }\res{t^{-2}b}\in{\mult{K}}^2,\\
		t^2\roi_{F} & b\in t^3\roi_{F},
	\end{cases}
\end{align*}
where we use Hensel's lemma to take a square root in the third case. The associated residue field approximations in the final two cases are given by
\[\begin{array}{ccc}
\begin{CD}
b^{1/2}+t^2\roi_{F} @>X^2>>b+t^3\roi_{F} \\
@VVV  @VVV  \\
 K  @>>2\res{b^{1/2}t^{-1}}X>  K 
\end{CD}
&\,\,&
\begin{CD}
-b^{1/2}+t^2\roi_{F} @>X^2>>b+t^3\roi_{F} \\
@VVV  @VVV  \\
 K  @>>-2\res{b^{1/2}t^{-1}}X>  K 
\end{CD}\\
&&\\
\begin{CD}
t^2\roi_{F} @>X^2>>b+t^3\roi_{F} \\
@VVV  @VVV  \\
 K  @>>-\res{bt^{-3}}>  K .
\end{CD}
&&\\
\end{array}\]
Proposition \ref{proposition_integral_of_singular_sections} therefore implies that for $y\in F$,
\[\int^F\Phi_\sub{sing}(x,y)\,dx\\
	=\begin{cases}
		0 & y\notin t^{-1}\roi_{F}, \\
		0 & y\in t^{-1}\roi_{F}\mbox{ but }\res{ty}\notin K^2,\\
		\int_K f(0, -2\res{(yt)^{1/2}}u)\,du\,X^2\\+ \int_K f(0, 2\res{(yt)^{1/2}}u)\,du\,X^2 & y\in t^{-1}\roi_{F}\mbox{ and }\res{ty}\in{\mult{K}}^2,\\
		0 & y\in \roi_{F}.
	\end{cases}\]
Therefore $y\mapsto \int^F\Phi_\sub{sing}(x,y)\,dx$ is the lift of \[v\mapsto \int_K f(0, -2v^{1/2}u)+f(0, 2v^{1/2}u)\,du\,X^2\,\Char_{K^{\times 2}}(v)\] at $0,-1$.

This function of $F$ need not be integrable, but as in the previous remark, there is a good argument to suggest that it should be, and why its integral should be zero:

Indeed, the function on the residue field has the form \[J(v)=\sum_{\substack{\w\in K \\ \res{q}(\w)=v \\ \res{q}'(\w)\neq 0}}\int_K g(-\res{q}'(\w)u)\,du\] where $g$ is a Schwartz-Bruhat function on $K$. Now replace the integrand by $g(-\res{q}'(\w)u)\Char_{K}(\w)$ and appeal to the appendix to deduce \[\int J(v)dv=\int\int g(u)\Char_{K}(\w)\,d\w du.\] But arguing as in the proceeding remark, $\int_Kd\w=0$, and so $\int_K J(v)\,dv=0$. Of course, the argument is nonsense because $J$ is not integrable, but it should be after a suitable extension of the measure.
\end{example}
%
%
\section{Negative depth with $\res{q}$ purely inseparably}\label{section_purely_insep}

We maintain all notation introduced at the beginning of the previous section but drop the additional hypothesis that $\res{q}'$ is not everywhere zero. Instead, we now assume $K$ has positive characteristic $p$ and that $\res{q}(X)$ is purely inseparable.

Whereas in the previous section conjecture \ref{conjecture} could fail to hold because the integration theory is not yet sufficiently developed, causing functions not to be integrable, we will present a result now to show that if $\res{q}$ is purely inseparable then all required functions are integrable, but the conjecture is simply false!

First note that, in the notation of the previous section, $\res{q}'$ being everywhere zero implies $\Phi=\Phi_\sub{sing}$. Secondly, proposition \ref{proposition_integral_of_singular_sections} remains valid, so that $x\mapsto \Phi(x,y)$ is integrable for any $y\in F$ and we have an explicit description of its integral.

\begin{proposition}\label{proposition_R=-1_purely_insep}
Suppose $R=-1$. Then both repeated integrals of $\Phi$ are well-defined, but $f$ may be chosen so that \[\int^F\int^F\Phi(x,y)\,dxdy\neq\int^F\int^F\Phi(x,y)\,dydx.\]
\end{proposition}
\begin{proof}
Arguing exactly as in corollary \ref{corollary_R=-1} it follows that $\int^F\Phi_\sub{sing}(x,y)\,dx=0$ for all $y\in F$, and therefore $y\mapsto \int^F\Phi_\sub{sing}(x,y)\,dx=0$ is certainly integrable, with integral $0$. That is, \[\int^F\int^F\Phi(x,y)\,dxdy=0.\]

The $dydx$ integral of $\Phi$ was showed to make sense in lemma \ref{lemma_reduction} and have value \[\int^F\int^F\Phi(x,y)\,dxdy=\int_K\int_K f(u,v)\,dudv.\] To complete the proof simply choose $f$ to be any Schwartz-Bruhat function on $K\times K$ such that $\int_K\int_K f(u,v)dudv$ is non-zero.
\end{proof}

\begin{remark}
Of course, using the terminology of \ref{remark_unferocious_case}, we are essentially now studying the ramification of ferocious extensions.

Further, the integration theory of \cite{Morrow2008} is easily modified to allow integration on a complete discrete valuation field whose residue field is any infinite field equipped with discrete counting measure; this is an elementary form of motivic integration. In that situation one may ask similar question about changes of variables and Fubini's theorem; results are generally easier to prove and closer to the analogous results for a usual local field. If the residue field is perfect, then the pathologies exhibited in this section no longer exist.

Perhaps the most immediate relation between measure theory and imperfectness is that the set of $p^\sub{th}$ powers of $K$ have zero measure, in stark contrast with in a perfect field.
\end{remark}

\section{Summary and future work}\label{section_summary}

Let us summarise our main results on conjecture \ref{conjecture}. Given data $a_1,a_2,n_1,n_2,h$ for the conjecture, let $q$ be the associated normalised polynomial and $R$ the depth; then:

\begin{enumerate}
\item If $\deg h\,(=\deg q)\le 1$ then the conjecture is true (theorem \ref{theorem_linear_case}).
\item If $R\ge 0$ then the conjecture is true (theorem \ref{theorem_R_positive})
\item If $\res{q}'$ is no-where vanishing on $ K $ then the conjecture is true (corollary \ref{corollary_nowhere_singular}).
\item If $R=-1$ and $\res{q}$ is not purely inseparable, then the conjecture is true (lemma \ref{lemma_reduction} $+$ corollary \ref{corollary_R=-1}).
\item If $R<-1$ and $\res{q}$ is not purely inseparable, then $y\mapsto\int^F\Phi(x,y)\,dx$ may fail to be integrable and so the conjecture may fail; it appears that it is possible to increase the space of integrable functions so that the conjecture becomes true (remark \ref{remark_nonsense} $+$ example \ref{example_nonsense}).
\item If $R=-1$ but $\res{q}$ is purely inseparable, then the conjecture fails and would continue to fail even if we increased the scope of the integral (section \ref{section_purely_insep}).
\item If $R<-1$ but $\res{q}$ is purely inseparable, then similarly to case (v) calculations become difficult. We have not included examples, but in all cases which the author can explicitly evaluate, $\int^F\int^F\Phi(x,y)\,dxdy=0$, Thus the conjecture seems to fail as in (vi).
\end{enumerate}

There are some immediate directions which demand study:

\begin{enumerate}
\item The integral must be extended to a wider class of functions so that the nonsense manipulations in remark \ref{remark_nonsense} and example \ref{example_nonsense} become valid. We should consider more general changes of coordinates on $F\times F$ than $(x,y)\mapsto (x,y-h(x))$. Similar methods to those in this paper will be required: firstly one needs to approximate the transformation at the level of $K\times K$ and find a suitable decomposition. This will lead to integrals over $K$ which can be explicitly evaluated as well as some functions on $F$; these functions on $F$ will perhaps be within the scope of the integral, or instead will provide further impetus for extending the integral.

\item The second issue concerns ramification. The imprecise relations between integration and ramification (it seems to be more useful to take the point of view of monodromy) which we have had occasion to mention should be made concrete. Perhaps this will be possible through a suitable theory of two-dimensional local zeta functions, just as the one-dimensional theories of local zeta functions and Igusa zeta functions capture ramification data of abelian extensions and monodromy of singularities. An enjoyable introduction to some of these topics, in dimension one, may be found in \cite{Nicaise2009}
\end{enumerate}

Also, the following seems worth consideration:

\begin{remark}
Let us return to the situation of remark \ref{remark_ramification_theory_in_perfect_case}, supposing that $L/M$ is a finite, Galois extension of complete, discrete valuation fields with perfect residue fields. The definition of the Hasse-Herbrand function implies that \[\frac{d\psi_{L/M}}{da}(a)=e_{L/M}^{-1}|G^a|^{-1},\] at least away from the ramification breaks, and therefore that \[\psi_{L/M}(a)=e_{L/M}^{-1}\int_{-1}^a|G^x|^{-1}\,dx-1,\] since both sides are $=-1$ at $a=-1$. But $|G:G^x|=|\pi_0(X_{L/M}^{x+1})|$, and so \[\psi_{L/M}(a)=f_{L/M}^{-1}\int_0^{a+1}|\pi_0(X_{L/M}^x)|\,dx-1\tag{$\ast$}\] for all $a\ge -1$.

If we think of `the number of connected components' as a measure, then ($\ast$) is a repeated integral taken over certain fibres, and it is exactly the variation of the fibres which contributes to the interesting structure of the Hasse-Herbrand function. Whether this repeated integral interpretation of ramification can be more systemically exploited in the local setting is an interesting question.
\end{remark}

\begin{appendix}
\section{Evaluation of an important integral on $K$}

Let $K$ be a local field, $f$ a Fubini function of $K\times K$, and $\psi\in K[X]$ a polynomial with $\psi'$ not everywhere zero. We discuss the function of $K$ given by \[J(v)=\sum_{\substack{\w\in K \\ \psi(\w)=v \\ \psi'(\w)\neq 0}}\int_K f(\w,-\psi'(\w)u)\,du.\] Note that the assumption that $f$ is Fubini implies that $J$ is defined (i.e. not infinite) for all $v$. We will prove the following:

\begin{proposition}
The function $J$ is integrable on $K$, with \[\int_K J(v)\,dv=\int_K\int_K f(\w,u)\,d\w du.\]
\end{proposition}
\begin{proof}
The proof is an exercise in analysis over a local field. Let $\Sigma=\{x:\psi'(x)=0\}$ be the finite set of singular points of $\psi$.

Let $v_0\in K$ and assume that there is a non-singular solution to $\psi(Y)=v_0$. The inverse function theorem for complete fields (see e.g. \cite{Igusa2000}) implies that there exists an open disc $V\ni v_0$, open discs $\Omega_1,\dots,\Omega_n$, and $K$-analytic maps $\lambda_i:V\to \Omega_i$, $i=1,\dots,n$ (that is, representable in $V$ by a convergent power series) with the following properties:
\begin{enumerate}
\item $\Omega_1,\dots,\Omega_n$ are pairwise disjoint;
\item $\psi(\Omega_i)=V$ for each $i$; moreover, $\psi|_{\Omega_i}$ and $\lambda_i$ are inverse diffeomorphisms between $\Omega_i$ and $V$;
\item the non-singular solutions in $K$ to $\psi(Y)=v_0$ are $Y=\lambda_1(v_0),\dots,\lambda_n(v_0)$.
\end{enumerate}
Moreover, we claim that, possibly after shrinking the sets $\Omega_i$, $V$, we may further assume
\begin{itemize}
\item[(iv)] for any $v\in V$, the non-singular solutions in $K$ to $\psi(Y)=v$ are $Y=\lambda_1(v),\dots,\lambda_n(v)$.
\end{itemize}
For if not, then there would exist a sequence $(x_n)_n$ in $K$ such that $x_n\notin\bigcup_i\Omega_i$ for all $n$ and $\psi(x_n)\to v_0$; the relative compactness of $\psi^{-1}(V)$ now allows us to pass to a convergent subsequence of $(x_n)$, giving an element $x\in K\setminus\bigcup_i \Omega_i$ which satisfies $\psi(x)=v_0$. But this contradicts (iii) and so proves our claim. Informally, the $\lambda_i$ parametrise the non-singular solutions of $\psi(Y)=v$, for $v\in V$.

For $v\in V$, we deduce that \[J(v)=\int \sum_{i=1}^n f(\lambda_i(v),-\psi'(\lambda_i(v))u)\,du\] and so
\begin{align*}
	\int_VJ(v)\,dv
	&=\sum_{i=1}^n\int_K\int_V  f(\lambda_i(v),-\psi'(\lambda_i(v))u)\,dvdu\\
	&=\sum_i\int_K\int_V\abs{\psi'(\lambda_i(v))}^{-1}f(\lambda_i(v),u)\,dvdu\\
	&=\sum_i\int_K\int_{\Omega_i} f(\w,u)\,d\w du\\
	&=\int_K\int_{\psi^{-1}(V)} f(\w,u)\,d\w du
\end{align*}
by Fubini's theorem and an analytic change of variables $v=\psi(\w)$. An elementary introduction to change of variables in integrals over non-archimedean fields may be found in \cite{Vladimirov1994}.

If $J$ is replaced by $J\Char_{A}$ for any measurable subset $A\subseteq V$ then this working is easily modified to show \[\int_AJ(v)\,dv=\int_K\int_{\psi^{-1}(A)} f(\w,u)\,d\w du\tag{$\ast$}.\]

It is now easy to see that $\psi(K\setminus \Sigma)$ admits a partition into countably many Borel sets $(A_j)_{j=1}^{\infty}$ where ($\ast$) holds with $A_j$ in place of $A$ for each $A$. Therefore
\begin{align*}
	\int_K J(v)\,dv
	&=\sum_j\int_{A_j} J(v)\,dv\\
	&=\sum_j\int_K\int_{\psi^{-1}(A_j)} f(\w,u)\,d\w du\\
	&=\int_K\int_{\Omega}f(\w,u)\,d\w du
\end{align*}
where $\Omega=\psi^{-1}(\psi(K\setminus\Sigma))=K\setminus \psi^{-1}(\psi(\Sigma))$ differs from $K$ only by a finite set. So we have reached the desired result: \[\int_K J(v)\,dv=\int_K\int_Kf(\w,u)\,d\w du.\]
\end{proof}
\end{appendix}

\affiliationone{
	Matthew Morrow\\
	Maths and Physics Building,\\
	University of Nottingham,\\
	University Park,\\
	Nottingham\\
	NG7 2RD\\
	United Kingdom\\
   	\email{matthew.morrow@maths.nottingham.ac.uk}}

\begin{thebibliography}{9}
\bibitem{Abbes2002}
{\sc A.~Abbes and T.~Saito} {`Ramification of local fields with imperfect residue fields'}
 Amer. J. Math.  124  (2002),  no. 5, 879-920.

\bibitem{Abbes2003}
{\sc A.~Abbes and T.~Saito}, {`Ramification of local fields with imperfect residue fields {I}{I}'}
Kazuya Kato's fiftieth birthday. Doc. Math.  2003,  Extra Vol., 5--72 (electronic).

\bibitem{Borger2004a}
{\sc J.~M.~Borger}, {`Conductors and the moduli of residual perfection'}. {\em Math. Ann.}, 329(1):1--30, 2004.

\bibitem{Borger2004}
{\sc J.~M.~Borger}, {`A monogenic {H}asse-{A}rf theorem'}. {\em J. Th\'eor. Nombres Bordeaux}, 16(2):373--375, 2004.

\bibitem{Fesenko2002}
{\sc I.~B.~Fesenko and S.~V.~Vostokov}, {`Local fields and their extensions'}, volume 121 of {\em
  Translations of Mathematical Monographs}. American Mathematical Society, Providence, RI, second edition, 2002. With a foreword by I. R. Shafarevich.

\bibitem{Fesenko2003}
{\sc I.~Fesenko}, {`Analysis on arithmetic schemes {I}'}, Doc. Math.,
  (2003), pp.~261--284 (electronic).
\newblock Kazuya Kato's fiftieth birthday.
\newblock available at {\tt http://www.maths.nott.ac.uk/personal/ibf/}

\bibitem{Fesenko2006}
{\sc I.~Fesenko}, {`Measure, integration and elements of harmonic analysis
  on generalized loop spaces'}, in Proceedings of the St. Petersburg
  Mathematical Society. Vol. XII, vol.~219 of Amer. Math. Soc. Transl. Ser. 2,
  Providence, RI, 2006, Amer. Math. Soc., pp.~149--165.

\bibitem{Frenkel2007}
{\sc E.~Frenkel}, {`Lectures on the Langlands program and conformal field theory'.} {\tt arXiv:hep-th/0512172v1}, 2005.

\bibitem{Hrushovski2006}
{\sc E.~Hrushovski and D.~Kazhdan}, {`Integration in valued fields'}, in
  Algebraic geometry and number theory, vol.~253 of Progr. Math., Birkh\"auser
  Boston, Boston, MA, 2006, pp.~261--405.

\bibitem{Hrushovski2008}
{\sc E.~Hrushovski and D.~Kazhdan}, {`The value ring of geometric motivic
  integration and the {I}wahori {H}ecke algebra of $\mbox{SL}_2$'}, {\tt
  arXiv:math.AG/0609115},  (2006).

\bibitem{Igusa2000}
{\sc J.~Igusa}, {\em An Introduction to the Theory of Local Zeta Functions}, AMS/IP Studies in Mathematics, 14, American Mathematical Society, Providence, RI; International Press, Cambridge, MA, 2000.

\bibitem{Johnson2000}
{\sc G.~W.~Johnson and M.~L.~Lapidus}, {\em The Feynman Integral and Feynman's
  Operational Calculus}, Oxford University Press, 2000.

\bibitem{Kato1989}
{\sc K.~Kato}, {`Swan conductors for characters of degree one in the imperfect residue field case'}. In {\em Algebraic {$K$}-theory and algebraic number theory ({H}onolulu, {HI}, 1987)}, volume~83 of {\em Contemp. Math.}, pages 101--131. Amer. Math. Soc., Providence, RI, 1989.

\bibitem{Kato1994}
{\sc K.~Kato}, {`Class field theory, {$D$}-modules, and ramification on higher-dimensional schemes. {I}'}. {\em Amer. J. Math.}, 116(4):757--784, 1994.

\bibitem{Lee2005}
{\sc H.~H.~Kim and K.-H.~Lee}, {`An invariant measure on $\mbox{GL}_n$ over
  $2$-dimensional local fields'}, University of Nottingham Mathematics preprint
  series, 2005.

\bibitem{Morrow2008}
{\sc M.~Morrow}, {`Integration on valuation fields over local fields'}, {\tt arXiv:math.NT/0712.2172}, awaiting publications with the Tokyo Journal of Mathematics, 2008.

\bibitem{Morrow2008a}
{\sc M.~Morrow}, {`Integration on product spaces and $GL_n$ of a valuation fields over a local field'}, Communications in Number Theory and Physics 2, 3, 563--592, 2008.

\bibitem{Morrow2009}
{\sc M.~Morrow}, {`An explicit approach to residues on and canonical sheaves of arithmetic surfaces'}, {\tt arXiv:math.NT/0911.0590}, submitted to the New York Journal of Mathematics, 2009.

\bibitem{Morrow_thesis}
{\sc M.~Morrow}, {`Investigations in two-dimensional arithmetic geometry'}, Ph.D. thesis at The University of Nottingham, available at {\tt http://www.maths.nottingham.ac.uk/personal/pmzmtm/}.

\bibitem{Neukirch1999}
{\sc J.~Neukirch},{`Algebraic number theory'}, volume 322 of {\em Grundlehren der Mathematischen Wissenschaften [Fundamental Principles of Mathematical Sciences]}. Springer-Verlag, Berlin, 1999. Translated from the 1992 German original and with a note by Norbert Schappacher, With a foreword by G. Harder.

\bibitem{Nicaise2009}
{\sc J.~Nicaise}, {`An introduction to $p$-adic and motivic zeta functions and the monodromy conjecture'}. {\tt arXiv:math.AG/0901.4225}, 2009.

\bibitem{Robinson1977}
{\sc A.~Robinson}, {\em Complete theories}, North-Holland, Amsterdam, 1956

\bibitem{Rudin1987}
{\sc W.~Rudin}, {\em Real and complex analysis}, McGraw-Hill Book Co., New
  York, third~ed., 1987.

\bibitem{Vladimirov1994}
{\sc V.~S.~Vladimirov, I.~V.~Volovich and E.~I.~Zlelnov}, {\em P-adic Analysis and Mathematical Physics}, World Scientific Publishing Co Pte Ltd, Series on Sovient \& East European Mathematics volume 1, 1992.

\bibitem{Xiao2007}
{\sc L.~Xiao}, {`Ramification theory for local fields with imperfect residue fields'}, available at {\tt http://www-math.mit.edu/~lxiao/}, 2007.

\bibitem{Xiao2008a}
{\sc L.~Xiao}, {`On ramifcation fltrations and p-adic diferential equations {I}: equal characteristic case'}. {\tt arXiv:0801.4962v3} (2008).

\bibitem{Xiao2008b}
{\sc L.~Xiao}, {`On ramifcation fltrations and p-adic diferential equations, {I}{I}: mixed characteristic case'}. {\tt arXiv:0811.3792} (2008).

\bibitem{Zhukov2000}
{\sc I.~Zhukov}, {`An approach to higher ramification theory'}, in Invitation
  to higher local fields (M\"unster, 1999), vol.~3 of Geom. Topol. Monogr.,
  Geom. Topol. Publ., Coventry, 2000, pp.~143--150 (electronic).

\bibitem{Zhukov2003}
{\sc I.~B.~Zhukov}, {`On ramification theory in the case of an imperfect
  residue field'}, Mat. Sb., 194 (2003), pp.~3--30.
\end{thebibliography}
\end{document}